\documentclass[nolineno,a4paper,USenglish,cleveref,autoref,thm-restate]
{socg}
\usepackage[utf8]{inputenc}
\usepackage{amsmath,amssymb}
\usepackage[textsize=tiny]{todonotes}
\usepackage{enumerate}
\usepackage{xspace,xcolor,soul}
\usepackage{booktabs,tabularx,colortbl,adjustbox}
\usepackage{comment}
\usepackage{graphicx}
\usepackage{paralist}
\usepackage{subcaption}
\usepackage{cleveref}
\usepackage[misc,geometry]{ifsym}

\DeclareMathOperator{\page}{pn}

\newcommand{\orcidID}[1]{\href{#1}{\includegraphics[scale=.03]{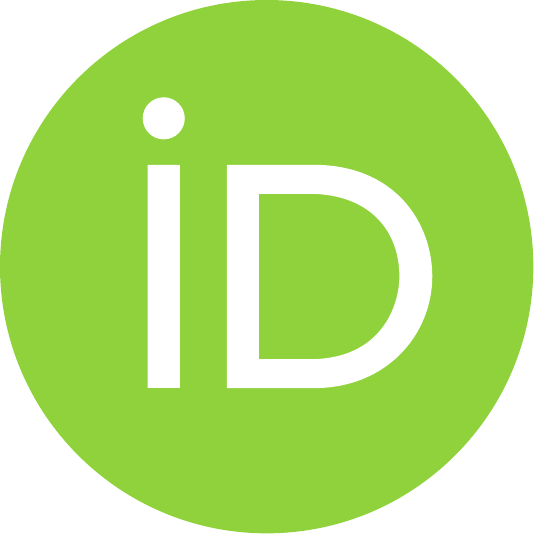}}}

\newtheorem{problem}{Problem}
\newtheorem{question}{Question}

\crefname{invariant}{Invariant}{Invariants}
\Crefformat{figure}{#2Fig.~#1#3}

\hideLIPIcs

\title{On the maximum number of edges of outer k-planar graphs}
\titlerunning{On the maximum number of edges of outer k-planar graphs}
\Copyright{Maximilian Pfister} 
\authorrunning{Maximilian Pfister} 


 \author{Maximilian Pfister}{Universit{\"a}t T{\"u}bingen}{maximilian.pfister@uni-tuebingen.de}{https://orcid.org/0000-0002-7203-0669}{This research was supported by the DFG grant SCHL 2331/1-1}

\usepackage{thmtools, thm-restate}

\graphicspath{{figures/}}

\makeatother

\begin{document}
\maketitle
\begin{abstract}
We study the maximum number of straight-line segments connecting $n$ points in convex
position in the plane, so that
each segment intersects at most $k$ others. This question can also be framed as the maximum number of edges of an outer $k$-planar graph on $n$ vertices. We outline several approaches to tackle the problem with the best approach yielding an upper bound of $(\sqrt{2}+\varepsilon)\sqrt{k}n$ edges (with $\varepsilon \rightarrow 0$ for sufficiently large $k$).
We further investigate the case where the points are arbitrarily bicolored and segments always connect two different colors (i.e., the corresponding graph has to be bipartite). To this end, we also consider the maximum cut problem for the circulant graph $C_n^{1,2,\dots,r}$ which might be of independent interest.
\end{abstract}

\section{Introduction}
A geometric graph $G =(P, E)$ is a drawing of a graph in the plane where the vertex set is
drawn as a point set $P$ in general position (that is, no three points are collinear) and each
edge of $E$ is drawn as a straight-line segment between its vertices. A geometric graph $G$ is called (i) \emph{convex}, if the pointset $P$ is convex and (ii) plane, if no two of its edges cross (that is, share a point in their relative interior). Similarly, $G$ is $k$-plane for some $k \geq 0$, if each edge crosses at most $k$ other edges.
In the following, with a slight abuse of notation, we will refer to convex geometric $k$-plane graphs as outer $k$-planar graphs (since the latter family of graphs always admits a convex geometric $k$-plane drawing). 
We are concerned with the following question:
\begin{question}\label{question:core}
What is the maximum number of edges an outer $k$-planar graph on $n$ vertices can have?
\end{question}
or $k=0$ these graphs are simply called outerplanar graphs, which have a rich history in many directions. By applying Euler's Polyhedra Formula, one can easily derive an upper bound of $2n-3$ edges for an $n$-vertex outerplanar graph. \cref{question:core} was considered in detail for small values of $k$: for $k=1$, \cite{DBLP:journals/algorithmica/AuerBBGHNR16} established a tight upper bound of $\frac{5n}{2}-4$ edges. For $k \in \{2,3,4\}$ the work of \cite{DBLP:conf/compgeom/AichholzerOOPSS22} and \cite{ábrego2024bookcrossingnumberscomplete} provided the bounds in the corresponding column of \cref{table:densities}.
Regarding general $k$, the conference version of \cite{chaplick2017beyond} allegedly showed that outer $k$-planar graphs are $\sqrt{4k+1}+1$ degenerate - which would immediately yield an upper bound of $(\sqrt{4k+1}+1)n \approx 2\sqrt{k}n$ edges. Unfortunately, this proof contained an error --- the corrected bound~\cite{DBLP:journals/corr/abs-1708-08723} yields a weaker upper bound of $\lfloor3.5\sqrt{k}\rfloor n$.
Aichholzer et. al.~\cite{DBLP:conf/compgeom/AichholzerOOPSS22} used the results for small values of $k$ together with the famous Crossing Lemma tailored to the convex setting to derive an upper bound of $\sqrt{\frac{243}{40}}\approx 2.465\sqrt{k}n$ edges, which was the previous best bound.
Closely related to our problem is the problem of the maximum number of straight-line segments connecting $n$ points in convex position, such that no $k$ segments pairwise intersect. Graphs that admit such a drawing are called outer $k$-quasiplanar graphs. The maximum number of outer $k$-quasiplanar graphs was already settled in 1992~\cite{quasiplanar}.
Recently, \cite{DBLP:conf/iwoca/Antic24} showed that any outer $k$-planar graph is outer ($k+1$)-quasiplanar --- as the bound of~\cite{quasiplanar} is linear in $k$, this result is not useful regarding \cref{question:core}.
\begin{figure}[t]
		\begin{subfigure}[b]{.45\textwidth}
		\centering
		\includegraphics[width=\textwidth,page=1]{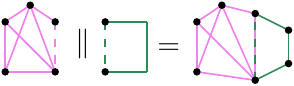}
		\subcaption{}
		\label{fig:concat}
	\end{subfigure}
 	\hfil
 	\begin{subfigure}[b]{.45\textwidth}
		\centering
		\includegraphics[width=\textwidth,page=2]{figures/first-fig.pdf}
		\subcaption{}
		\label{fig:outercopy}
	\end{subfigure}
 	\caption{(a) Illustration of the \emph{concatenation} operation. (b) Illustration of the \emph{outercopy} operation}
  \end{figure} 

\paragraph{Our Contribution}
In \cref{sec:general}, we consider \cref{question:core} for general outer $k$-planar graphs. We first recall the best known lower bound construction, before we attack the upper bound from a total of four different directions, i.e., 
\begin{itemize}[a)]
\item using known results from (non-convex) topological $k$-planar graphs,
\item finding bounds for small values of $k$ and applying the Crossing Lemma,
\item bounding the maximum minimum degree of outer $k$-planar graphs and
\item directly tackling the density problem.
\end{itemize}
We then derive an auxiliary result in \cref{sec:circulant_graphs} 
by considering the maximum cut problem for so called circulant graphs $C_n^{1,2,\dots,r}$, which will be used in \cref{sec:bipartite} where we apply our methods to bipartite outer $k$-planar graphs.
Finally, we conclude in \cref{sec:open-problems} with open problems raised by our work and future research directions.

\section{Preliminaries}\label{sec:prelim}
Throughout the paper, we will assume that $G=(V,E)$ with $|V| = n$ is an outer $k$-planar graph. With a slight abuse of notation, we will also use $V$ to refer to the underlying pointset of $G$. Denote by $C[G] \subset E$ the edges of $G$ which are part of the convex hull of $V$. Any edge of $E \setminus C[G]$ is called a \emph{diagonal} of $G$.
For a diagonal $(a,b)$ consider the induced line $L_{ab}$ and denote by $H_{ab}^-$ ($H_{ab}^+$) the minimum (maximum) of the two open half-planes with respect to the number of vertices of $G$ they contain. If $H_{ab}^-$ contains $t$ vertices, then $(a,b)$ \emph{splits off} $t$ vertices. In this case, the \emph{length} of $(a,b)$ is $t$.

\noindent
In the following, we will define two operations on outer $k$-planar graphs which will be useful for the remainder of the paper.
Let $G_1$ and $G_2$ be two outer $k$-planar graphs. Then $G = G_1 \mathbin\Vert G_2$ will be called the \emph{concatenation} of $G_1$ and $G_2$ such that, $V[G] = V[G_1] \cup V[G_2] \setminus \{u,v\}$ and $E[G] = E[G_1]\cup E[G_2] - (u,v)$, where $(u,v)$ is an edge of $C[G_1]$. Informally speaking, we will identify an edge of $C[G_2]$ with an edge $(u,v)$ of $C[G_1]$ and join the two graphs together, see \cref{fig:concat}. This procedure is also known as the \emph{clique-sum} operation for a clique of size two.

\noindent
Let $G$ be an outer $k$-plane graph. The \emph{outercopy} $G' = (V',E')$ of $G$ with $V' = V$, $E' = E[G] \cup \{e^1 | \quad  e \in E \setminus C[G]\} $ is obtained by copying every edge of $G$ which does not belong to $C[G]$. In particular, observe that since $|C[G]| \leq n$, we have that $|E'| \geq 2|E| - n$. Further observe that $G'$ is a $k$-planar non-homotopic (multi)-graph, i.e., there exists a drawing $\Gamma'$ of $G'$ in the plane where its vertices are mapped to points and its edges are mapped to Jordan curves such that any curve intersects at most $k$ others - $\Gamma'$ is easily derived by placing the vertices of $G'$ on the point set $V$, by drawing all edges of $G'$ which are also contained in $G$ as straight-line segments, while the remaining edges of $G' \setminus G'$ are drawn outside of the region bounded by the convex hull of $V$, see \cref{fig:outercopy}. By construction, any pair of multiedges enclose some part of $C[G]$ that contains at least one vertex, hence the resulting graph is non-homotopic.

\section{Outer k-planar graphs}\label{sec:general}
\subsection{Lower bound}
Let us recall the current best lower bound construction due to \cite{ábrego2024bookcrossingnumberscomplete} restated using our terminology.
\begin{theorem}[\cite{ábrego2024bookcrossingnumberscomplete}]\label{thm:lower-general}
    For infinitely many values of $k$, there exists outer $k$-planar graphs on $n$ vertices with $\sqrt{k}n+\Theta(1)$ edges. 
\end{theorem}
\begin{proof}
Choose $k = (\frac{x-2}{2})^2$ for $x \in \mathbb{N}$. Then, the convex drawing of $K_x$ is $k$-plane.
Define $G = K_x \mathbin\Vert K_x \mathbin\Vert \dots \mathbin\Vert K_x$.
Assuming $n-2 \mod (x-2) = 0$, graph $G$ has exactly
\[\frac{n-2}{x-2} \binom{x}{2} - (\frac{n-2}{x-2}-1) =\frac{1}{2}n(x + 1) - x - 2 = \frac{1}{2}n(2\sqrt{k}+3) - x - 2 = n\sqrt{k}+3n-2\sqrt{k} \]
edges which concludes the proof.
\end{proof}
With a (decent) lower bound at hand, we will in the following attack the upper bound using three different approaches.
\subsection{The lazy approach: Leveraging known results from the non-convex case} \label{sec:crossing-lemma}
Before we state the main proof, we will derive the following technical one.
\begin{lemma}\label{lem:edge-density-mult-two}
Let $G=(V,E)$ be a $k$-planar multigraph with maximum edge multiplicity two and $|E| > 6.77|V|$. Then $G$ has at most $5.243\sqrt{k}n$ edges.
\end{lemma}
\begin{proof}
    The result of Szekely~\cite{DBLP:journals/cpc/Szekely97} states that the number of crossings for a multigraph with maximum edge multiplicity $x$ is at least 
    \[cr(G) \geq c'\frac{m^3}{xn^2} \]
    Using $x = 2$ and the current best bound of $c' = \frac{1}{27.48}$~\cite{BestConstantCL}, which holds for $|E| > 6.77|V|$, we obtain
    \[cr(G) \geq \frac{1}{27.48}\frac{m^3}{2n^2} \]
    Since any edge is involved in at most $k$ crossings, we further have $cr(G) \leq \frac{km}{2}$.
    Solving for $m$ then yields the desired result.
\end{proof}

\begin{theorem}[(First variant)]
For every $k \geq 5$, every outer $k$-planar graph $G$ on $n$ vertices has at most $2.85\sqrt{k}n$ edges.
\end{theorem}
\begin{proof}
Suppose for a contradiction that there exists an outer $k$-planar graph $G$ on $n$ vertices with $m > 2.85\sqrt{k}n$ edges.
Let $G'$ be the outercopy of $G$.
First observe that the maximum edge multiplicity in $G'$ is two. Further, by construction, we have 
\[|E[G']| \geq 2m-n \geq 5.7\sqrt{k}n - n > 6.77n \]
as $k \geq 5$ holds.
Hence Lemma~\ref{lem:edge-density-mult-two} has to hold for $G'$, but we have
\[|E(G')| > 2\cdot(2.85\sqrt{k}n)-n = 5.243\sqrt{k}n +(0.45\sqrt{k} - 1)n \geq 5.243\sqrt{k}n\]
since $k\geq 5$ and thus $0.45\sqrt{k} \geq 1$ and we obtain a contradiction.
\end{proof}

\subsection{The common approach: Finding bounds for small values of k}\label{sec:small-values}
Over the years, the most successful approach to refine the upper bound on the maximum number of edges of $k$-planar graphs was to find tight bounds for small values of $k$ (which was achieved for $k \leq 4)$, which improves the famous Crossing Lemma which can then be used to improve the maximum edge density bound.
The authors of \cite{DBLP:conf/compgeom/AichholzerOOPSS22} showed that convex geometric $k$-plane graphs with $n$ vertices have at most $(\frac{k+4}{2}-(k+3))n$ edges if $k \leq 4$.
Plugging these results into the Crossing Lemma yields
\[cr_o(G) \geq \frac{20}{243}\frac{m^3}{n^2}\] and consequently an upper bound on the number of edges of 
\[m \leq \sqrt{\frac{243}{40}k}n\]
for outer $k$-planar graphs.

In order to derive the upper bound of $(\frac{k+4}{2}-(k+3))n$ for small values of $k$, the authors used a counting argument on the number of half-edges based on a technical lemma of \cite{DBLP:journals/combinatorica/PachT97}.
The results of \cite{ábrego2024bookcrossingnumberscomplete} were derived by a quite involved case analysis. 
We provide a very simple argument (by incorporating even more past results) which matches the tight bounds for $k \leq 3$ and arguably for $k=4$, see the remark.
\begin{center}
\begin{table}
\caption{Lower and upper bounds on the number of edges of convex geometric $k$-plane graphs for small $k$}
\begin{tabular}{ c | c | c | c | c  } \label{table:densities}
 $k$ & Lower bound~\cite[Thm. 10]{ábrego2024bookcrossingnumberscomplete} &  \cite[Prop. 5]{DBLP:conf/compgeom/AichholzerOOPSS22} & \cite[Thm. 4]{ábrego2024bookcrossingnumberscomplete} & \cref{lem:upper-bound-small-k} \\ 
 \hline
 0 & $2n-3$& $2n-3$ & $2n-3$ & $2n-3$ \\  
 1 & $2.5n-4$& $2.5n-4$ & $2.5n-4$ & $2.5n-4$ \\   
 2 & $3n-5$& $3n-5$ & $3n-5$ & $3n-5$ \\  
 3 & $3.25n-6$& $3.5n-6$ & $3.25n-6$ & $3.25n-6$ \\ 
 4 & $3.5n-6$& $4n-7$ & $3.5n-6$  & $(3.5n-6)$ \\  
\end{tabular}
\end{table}
\end{center}
\begin{lemma}\label{lem:upper-bound-small-k}
The number of edges of an $n$-vertex outer $k$-planar graph with $k \leq 3$ is bounded according to Table~\ref{table:densities}.
\end{lemma}
\begin{proof}
For $k \leq 3$, let $G$ be an outer $k$-planar graph with $n$ vertices and the maximum number of edges.
Let $G'$ be the outercopy of $G$ with corresponding $k$-plane drawing $\Gamma'$. Recall that $G'$ is an $n$-vertex non-homotopic $k$-planar multigraph by construction and hence as at most $\{3n-6\footnote{folklore},4n-8~\cite{DBLP:journals/dm/SoneS23},5n-10~\cite{DBLP:conf/compgeom/Bekos0R17},5.5n-11~\cite{DBLP:conf/gd/Bekos0R16}\}$ edges for $k \in \{0,1,2,3\}$, together with the observation that $|E'| \geq 2|E|-n$, it follows that $|E| \leq
\{2n-3,2.5n-4,3n-5,3.25n-6\}$ as desired.

\end{proof}
\textit{Remark:} It is mentioned as a remark in~\cite{DBLP:journals/comgeo/Ackerman19} without an explicit proof that the upper bound of $6n-12$ for the number of edges of $n$-vertex  $4$-planar graphs also holds for a special kind of non-homotopic multigraphs (which our construction would satisfy) - this would yield the desired bound of $3.5n-6$ for the case of $k=4$. 

Using these bounds, we can, analogously to \cite{DBLP:conf/compgeom/AichholzerOOPSS22}, derive a slightly stronger constant for the convex crossing lemma - we omit the details as they are analogous to the one of \cite{DBLP:conf/compgeom/AichholzerOOPSS22} besides some minor numerical differences.
\begin{lemma}\label{lem:cr-lemma-v1}
Let $G$ be a graph with $n$ vertices and $m$ edges such that $m \geq \frac{171}{40}n$. The outer drawing of $G$ has at least
\[cr_o(G) \geq \frac{8000}{87723}\frac{m^3}{n^2}\]
crossings.
\end{lemma}
\begin{theorem}[(Second variant)]\label{thm:edge-density-second}
Every outer $k$-planar graph $G$ on $n$ vertices has at most $\sqrt{\frac{87723}
{16000}k}n \approx 2.34\sqrt{k}n$ edges.
\end{theorem}
\begin{proof}
For $k \leq 4$, the bounds of Table~\ref{table:densities} are strictly better, hence there is nothing to prove. Assume $k \geq 5$ and observe that $\sqrt{\frac{87723}
{16000}k} \geq \sqrt{\frac{87723}
{3200}} > 5 > \frac{171}{40}$, thus the claim is shown as soon as $m < \frac{171}{40}n$. Hence, we consider the case $m \geq \frac{171}{40}$ and can use Lemma~\ref{lem:cr-lemma-v1} (while observing that any edge is part of at most $k$ crossings) to get
\[\frac{mk}{2} \geq cr_o(G) \geq \frac{8000}{87723}\frac{m^3}{n^2}\]
which yields the desired result after rearranging.
\end{proof}

\subsection{The local approach}
\label{sec:direct-approach}
While the previous two variants used results of the non-convex case, we will now attack the problem by bounding the maximum minimum degree an outer $k$-planar graph can have.
The following proof was suggested by an anonymous reviewer of an earlier draft\footnote{and, while being more elegant, improved on our previous result of $2.02\sqrt{k}$} and is a generalization of a solution to Problem 8 in Grade 8-9 of the 239 School Mathematical Olympiad in 2024, see~\cite{357c0fe24bc3431eac6c30a0afea69fe}.
As this result was unknown to us and apparently also to other previous work~\cite{DBLP:conf/compgeom/AichholzerOOPSS22,chaplick2017beyond,ábrego2024bookcrossingnumberscomplete}, we include it here for completeness.
\begin{theorem}\label{thm:general-maxmindeg} 
The maximum minimum degree of an outer $k$-planar graph is $2\sqrt{k+1}+2$.
\end{theorem}
\begin{proof}
Suppose for a contradiction there exist an outer $k$-planar graph $G$ whose maximum minimum degree is larger than $2\sqrt{k+1}+2$.
This implies that every vertex is incident to more than $2\sqrt{k+1}$ diagonals. Let's call a diagonal long if its length is at least $\sqrt{k+1}$, and short otherwise. Obviously, at least one long diagonal is incident to each vertex. Among all the long diagonal, let's choose the shortest one and call it $D$. We prove that it is crossed by at least $k+1$ other diagonals. Indeed, let its length be $l$. Let's number these $l$ vertices in the order of traversal of the polygon $v_1, v_2, \cdots, v_l$. Then $v_1$ is incident to at most $\sqrt{k+1}$ short diagonals which do not cross $D$, which means more than  $\sqrt{k+1}$ diagonals incident to $v_1$ cross $D$ (obviously, all long diagonals incident to $v_i$ intersect $D$ by our minimality assumption on $D$). Symmetrically, more than $\sqrt{k+1}$ diagonal intersecting $D$ are incident to $v_l$, more than $\sqrt{k+1}-1$ diagonals crossing $D$ are incident to $v_2$ or $v_{l-1}$, more than $\sqrt{k+1}-2$ such diagonals are incident to $v_3$ and $v_{l-2}$, and so on. For $l > 2\sqrt{k+1}$, we get more than
$$2(1 + 2 + \cdots + \sqrt{k+1})  > k+1$$
such diagonals. If $l < 2\sqrt{k+1}$, then from each vertex there are less than $l$ diagonal that do not cross $D$, which means more than $2\sqrt{k+1} -l$ diagonals cross $D$. In total, we have at least $l(2\sqrt{k+1}-l)$ diagonal intersecting $D$. But for the extreme vertices, the previously found numbers of diagonal - $\sqrt{k+1}, \sqrt{k+1}-1,$ etc. - are more than the universal estimate of $2\sqrt{k+1} - l$. Therefore, for such vertices, we introduce correction additions \[\sqrt{k+1} - (2\sqrt{k+1} - l) = l - \sqrt{k+1}, \sqrt{k+1} - 1 - (2\sqrt{k+1} - l) = l - \sqrt{k+1} - 1, \cdots, 1\]
In total, there will be more than
\[l(2\sqrt{k+1} - l)+2(1+2+\dots+ (l - \sqrt{k+1})) = k+1 + l - \sqrt{k+1} \geq k+1\]
diagonals intersecting $D$, and we obtain a contradiction.
\end{proof}
This bound on the degree is tight as witnessed by \cref{thm:lower-general}.
\begin{corollary}
    Every convex geometric $k$-plane graph can be colored with $\lfloor2\sqrt{k+1}\rfloor+1$ colors.
\end{corollary}

\begin{corollary}[(Third variant)]
Every outer $k$-planar graph $G$ on $n$ vertices has at most $(2\sqrt{k+1}+2)n$ many edges 
\end{corollary}

\subsection{The direct approach}
The proof is a generalization of a stackexchange comment~\cite{4767941} regarding a solution to the aforementioned math olympiad problem.
\begin{theorem} \label{thm:upper-new}
An outer $k$-planar graph has at most $(\sqrt{2}+\varepsilon)\sqrt{k}n$ diagonals with 
$\lim_{k\to \infty} \varepsilon= 0$.
\end{theorem}
\begin{proof}
Let $x = (\sqrt{2}+\varepsilon)$.
Let $G$ be a vertex minimum counterexample to \cref{thm:upper-new} and let $v_1,\dots,v_n$ be the vertices of $G$, where the order is defined by the outer face. Fix a parameter $1 \leq l_0 \leq \frac{n}{2}$ and choose a shortest diagonal of $G$ of length at least $l_0$. Denote the diagonal by $D$, assume its length is $l \geq l_0$ and w.l.o.g. assume that $D = v_1v_{l+3}$. Let $G_1$ and $G_2$ denote the two outer $k$-planar graphs obtained from $G$ by splitting along $D$ and removing all diagonals which intersect $D$. W.l.o.g. assume that $G_1$ contains $l+2$ vertices on its boundary, which implies that $G_2$ contains $n-l+2$ vertices.
Recall that by our choice of $D$, any diagonal contained in $G_1$ has length at most $l_0$.
Note that the length of a diagonal is meassured w.r.t. $G$, that is, without interpreting $D$ as a boundary edge. Thus, the number of diagonals of length two contained in $G_1$ is exactly $l$, since the two diagonals $v_{l+2}v_1$ and $v_{l+3}v_2$ do not have length two w.r.t. $G$. 
Likewise, the number of diagonals of length three is $l-1$, as the three diagonals $\{v_{l+1}v_1,v_{l+2}v_2,v_{l+3}v_3\}$ do not have length three w.r.t. $G$. In general, there exist at most $(l+2-p)$ many diagonals of length $p$ in $G_1$.
Hence, the number of diagonals in $G_1$ is at most
\[(l+2-2)+(l+2-3)+\dots(l+2-l_0) = (l_0-1)\frac{(2l-l_0+2)}{2}\]
Since $G$ is chosen as a minimum counterexample, it follows that the number of diagonals that are contained in $G_2$ is at most $x(n-l+2)$.
This implies that $D$ is crossed by at least
\[\underbrace{xn+1}_{\text{G counterexample}} -\underbrace{1}_D - \underbrace{x(n-l+2)}_{G_2} - \underbrace{(l_0-1)\frac{(2l-l_0+2)}{2}}_{\text{Diagonals in }G_1} = xl-2x - (l_0-1)\frac{(2l-l_0+2)}{2} \] edges.
In order to obtain a contradiction, set
\[xl-2x - (l_0-1)\frac{(2l-l_0+2)}{2} \stackrel{!}{>} k\]
\[\Leftrightarrow xl-2x-l_0l+0.5l_0^2-l_0+l-0.5l_0+1 \stackrel{!}{>}k\]

Choose $l_0 = \sqrt{2}\sqrt{k}$. Recall that $x = (\sqrt{2}+\varepsilon)\sqrt{k}$ and that $l \geq l_0$ holds. 
This implies that $xl-l_0l = \varepsilon\sqrt{k}l$. Since $0.5l_0^2 = k$, we have

\[k + 1 + \varepsilon\sqrt{k}l -2(\sqrt{2}+\varepsilon)\sqrt{k} -0.5\sqrt{2}\sqrt{k} \stackrel{!}{>}k \]
Thus, we have to assert
\[\varepsilon\sqrt{k}l + 1 \stackrel{!}{>} 2(\sqrt{2}+\varepsilon)\sqrt{k} +0.5\sqrt{2}\sqrt{k}\]
Since $l \geq l_0 = \sqrt{2}\sqrt{k}$, the LHS is at least $\varepsilon\sqrt{k}(\sqrt{2}\sqrt{k})+1$ and thus
\[(\sqrt{2}k-2\sqrt{k})\varepsilon \stackrel{!}{>}\frac{5\sqrt{2}}{2}\sqrt{k}-1 \]
Since the LHS side grows linearly in $k$, while the RHS has a sublinear growth, there exists, for any $\varepsilon > 0$, a sufficiently large $k$ for which the LHS is larger than the RHS as desired.
\end{proof}
\textit{Remark} If one wants to apply the formula for specific $k$, one can use the last inequality to properly choose $\varepsilon$.
For example, for $k=50$, it holds for $\varepsilon > 0.43$.
\begin{corollary}[(Fourth variant)]\label{cor:best-bound}
An outer $k$-planar graph has at most $(\sqrt{2}+\varepsilon)\sqrt{k}n+n$ edges. For sufficiently large $k$, $\varepsilon$ tends to $0$.
\end{corollary}

\section{On the maximum cut of circulant graphs}
\label{sec:circulant_graphs}

In the proofs of \cref{thm:general-maxmindeg} and of \cref{thm:upper-new}, we assumed in both cases that all of the  diagonals which are completely contained in the region delimited by our diagonal are in fact present. While this could occur in the general case, this is impossible once we turn to bipartite outer $k$-planar graphs as we have to avoid odd-length cycles. If the enclosed vertices are denoted by $v_1,\dots,v_x$ and the maximum length of a short diagonal is $r$, then maximizing the number of short edges of a bipartite subgraph on $v_1,\dots,v_x$ is analogous to the following one.
\begin{problem}\label{prob:one}
Let $S = s_1,\dots,s_x$ be a binary string of length $x$ and $r\in N^+$ be a given parameter. Find an upper bound for 
 \[\sum_{i=1}^n \sum_{j=-r}^r s_i \oplus s_{i+j}  \] 
 where $(i+j)$ is only considered in $[0,n-1]$.
 \end{problem}

By relaxing the boundary  condition on  $(i+j)$, we obtain
 \begin{problem}\label{prob:two}
 Let $S = s_1,\dots,s_n$ be a binary string of length $n$ and $r\in N^+$ be a given parameter. Find an upper bound for 
\[\sum_{i=1}^n \sum_{j=-r}^r s_i \oplus s_{i+j}  \] 
 with $(i+j) \mod n$.
 \end{problem}
 Clearly, the solution of \cref{prob:two} also yields an upper bound on \cref{prob:one}.

 We will provide an upper bound for \cref{prob:two} by converting it into the following analogous statement.
A \emph{circulant graph} is an undirected graph whose adjacency matrix is a circulant matrix, i.e., every row is composed of the same elements and each row is rotated one element to the right relative to the preceding row.
The easiest example of a circulant graph is the cycle graph $C_n$.
The circulant graph $C_n^{j_1,\dots,j_r}$ with jumps $j_1,\dots,j_r$ is defined as the graph with $n$ nodes labeled $0, \dots, n-1$ where each node $i$ is adjacent to $2r$ nodes $i\pm j_1,\dots,i \pm j_k \mod n$.
Hence, we can restate \cref{prob:two} as follows:
\begin{problem}
    Let $G = C_n^{1,2,\dots,r}$ and let $G' \subseteq G$  be the maximum bipartite (w.r.t. the number of edges) subgraph of $G$. Find an upper bound for $2|E(G')|$.
\end{problem}
This corresponds to (twice the value of) the maximum cut problem of $C_n^{1,2,\dots,r}$, which is (finally) the formulation which we will prove:
\begin{lemma} \label{lem:maxcut}
Let $G$ be the circulant graph $C_n^{1,2,\dots,r}$ and denote by $mc(G)$ the value of the maximum cut of $G$. Then $ mc(G) \leq (\frac{5r}{8}+76)n$. 
\end{lemma}
\begin{proof}
Before we start with the main part of the proof, let us first provide a trivial upper bound of $mc(G) \leq rn = |E(G)|$ by definition of $G$. In particular, observe that $(\frac{5r}{8}+76)n \geq rn$ holds for $r < 176$. Hence, for the remainder of the proof, assume that $r \geq 176$ holds.
A key element in the proof is the following result:
\begin{lemma}[\cite{Mohar1990}]
$mc(G) \leq \frac{1}{4}\lambda_{max}(L)\cdot n$, where $\lambda_{max}(L)$ is the largest eigenvalue of the Laplacian matrix $L$ of $G$.
\end{lemma}
In order to prove~\cref{lem:maxcut}, it is sufficient to show that $\lambda_{max}(L) \leq 2.5r+304$ holds.
First observe that since $G$ is $2r$-regular, the Laplacian eigenvalues are obtained by subtracting the adjacency matrix eigenvalues from $2r$. The eigenvalues of circulant matrices were studied in detail \cite{CIT-006} and we obtain the following formula for the $j-th$ eigenvalue:
\[\lambda_j =\sum_{k=0}^{n-1} c_k w^{jk}\] where $w= \exp(\frac{2\pi i}{n})$ is a primitive $n-th$ root of unity with $i$ being the imaginery unit. 
Observe first that $\lambda_0 = 2r$. We keep this value in mind and assume from now on that $j > 0$. Recall that in our case, $c_k$ is one for $k \in \{n-r,n-(r-1),\dots,n-1,1,2,\dots,r\}$ and otherwise zero. Thus
\[\lambda_j = w^{j}+w^{2j}+\dots+w^{rj}+w^{(n-r)j}+ w^{(n-r+1)j} + \dots + w^{(n-1)j} \]
Since $G$ is undirected, the eigenvalues of $G$ are real numbers and we can use $\cos(\frac{2\pi j}{n}) = Re(w^j) = Re(\exp(\frac{2\pi j i}{n}))$.
Let $p = \cos(\frac{2\pi j}{n})$ and observe that
\[p^{(n-x)} = p^n \cdot p^{-x} = p^{-x} = p^{x} \]
where the second equality follows since $p^{n} = 1$ and the third inequality follows since $cos(x) = cos(-x)$.
But then we can rewrite
\[\lambda_j = 2\sum_{k=1}^r p^k\]
Observe that $\lambda_j+1$ is the Dirichlet kernel $D_r(\frac{2\pi k}{n})$. 
The minimum value of a Dirichlet kernel was studied in \cite{Mercer}, who showed, for $r \geq 2$, a lower bound of $\min\{-\frac{5}{12},\frac{1}{r}+C_0-\frac{8\pi}{2(r+1)}\}\cdot r$ with $C_0 \geq -0.4344$.
Now, $f(r) = \frac{1}{r}+C_0-\frac{8\pi}{2(r+1)}$ is monotonically increasing for $r > 0$. Recall that $r \geq 176$ by assumption and observe that $f(176) \approx -0.4997 > -0.5$. It follows that $\lambda_j$ is at least $-0.5r - 1$ and hence $\lambda_{max}(L)$ is at most $2r-(-0.5r-1) = 2.5r+1$, which implies $mc(G) \leq \frac{1}{4}({2.5r+1})n = (\frac{5r}{8}+0.25)n \leq (\frac{5r}{8}+76)n$ as desired.
\end{proof}
\begin{corollary}\label{cor:balls}
Let $s = s_1,\dots,s_n$ be a binary string of length $n$ and $r\in N^+$ be a given parameter. Then,
\begin{equation}\label{eq:sum}
\sum_{i=1}^n \sum_{j=-r}^r s_i \oplus s_{i+j}  \leq (\frac{5r}{4}+152)n
\end{equation}
with $(i+j) \mod n$.
\end{corollary}

\section{Bipartite outer k-planar graphs}\label{sec:bipartite}
We are now ready to turn our attention to bipartite outer $k$-planar graphs.
As before, we begin with lower-bound constructions.

\subsection{Lower bounds}
The work of~\cite{TwoLayer}, which was concerned with the edge-density of $k$-plane graphs where the vertices are placed on two layers, yields a natural lower bound for the edge density of bipartite outer $k$-planar graphs, where the vertices of each partition are assumed to appear consecutively (in a single block) on the outer face:
\begin{theorem}[12,\cite{TwoLayer}]
There exist infinitely many bipartite convex geometric $k$-plane graphs (in the consecutive setting) with $n$ vertices and $m = \lfloor \sqrt{k/2}\rfloor n - \mathcal{O}(f(k)) \approx 0.707\sqrt{k}n - \mathcal{O}(f(k))$ edges.    
\end{theorem}

Let us now consider the other extrema, i.e., when traversing the outer face we encounter $a_1b_1a_2b_2\dots$ with $a_i$ in partition $A$ and $b_i$ in partition $B$. We will refer to this setting as \emph{alternating} and call $a_ib_i$ a \emph{pair}.
\begin{observation}\label{obs:max-bip-k-plane}
    In the alternating setting, the largest complete bipartite convex geometric $k$-plane graph is $K_{x,x}$ with $x = \sqrt{2k}+1$.
\end{observation}
\begin{proof}
    Consider a longest edge $e$ which splits exactly $\lfloor\frac{x-1}{2}\rfloor$ pairs of vertices to one side and $\lceil\frac{x-1}{2}\rceil$ vertices to the other side. Since we have two edges from every pair to every pair on the other side, the edge $e$ is crossed exactly
    \[k = 2\cdot\lfloor\frac{x-1}{2}\rfloor\lceil\frac{x-1}{2}\rceil\] times.
    Now, if $x$ is even, we obtain $x = \sqrt{2k+1}+1$, otherwise $x = \sqrt{2k}+1$.
\end{proof}

\begin{theorem}\label{thm:lower-bip}
There exist infinitely many bipartite convex geometric $k$-plane graphs (in the alternating setting) with $n$ vertices and $m = \frac{1}{\sqrt{2}}\sqrt{k}n + \theta(1) \approx 0.707\sqrt{kn}$ edges.   
\end{theorem}
\begin{proof}
$K_{x,x}$ with $x = \sqrt{2k}+1$ has $2\sqrt{2k}+2$ vertices, $2k+2\sqrt{2k}+1$ edges
and is $k$-plane by \cref{obs:max-bip-k-plane}. Let $G = K_{x,x}  \mathbin\Vert K_{x,x} \mathbin\Vert \dots \mathbin\Vert K_{x,x}$ ($l$ copies of $K_{x,x}$). Then, $n =  l(2\sqrt{2k})+2$ which yields 
\[m =l(2k+2\sqrt{2k})+1 = \frac{1}{\sqrt{2}}\sqrt{k}n - \sqrt{2k}+n-2 \] as desired.
\end{proof}
We will now turn our attention to the upper bounds.
\subsection{The lazy approach}
There exists a crossing lemma specifically tailored to bipartite graphs with the current best constant being $c = \frac{1024}{16875}~\cite{bungener_et_al:LIPIcs.GD.2024.28}$. Plugging this into the result of \cite{DBLP:journals/cpc/Szekely97}, we obtain
\begin{lemma}\label{lem:edge-density-mult-two-bip}
Let $G$ be a bipartite $k$-plane multigraph with maximum edge multiplicity two. Then $G$ has at most $4.06\sqrt{k}n$ edges.
\end{lemma}
\begin{theorem}[(First bipartite variant)]
    For every $k \geq 5$, every bipartite outer $k$-planar graph has at most $2.228\sqrt{k}n$ edges.
\end{theorem}
\begin{proof}
Suppose for a contradiction that there exists a bipartite outer $k$-planar graph $G$ on $n$ vertices with $m > 2.228\sqrt{k}n$ edges.
Let $G'$ be the outercopy of $G$.
Recall that $|E[G']| \geq 2m-n$.
By definition, the maximum edge multiplicity in $G'$ is two, thus~\cref{lem:edge-density-mult-two-bip} has to hold, but we have
\[|E(G')| > 2\cdot(2.228\sqrt{k}n)-n = 4.06\sqrt{k}n +(0.45\sqrt{k} - 1)n \geq 5.243\sqrt{k}n\]
since $k\geq 5$ holds by assumption and thus $0.45\sqrt{k} \geq 1$ and we obtain a contradiction.
\end{proof}

\subsection{The common approach}
 Simply adapting our proof strategy of \cref{sec:small-values} proves to be difficult for the bipartite case as we neither have tight bounds for bipartite $k$-plane \emph{non-homotopic multigraphs} for $k \in \{1,2,3\}$ nor direct bounds for bipartite outer $k$-plane graphs in the literature. Hence, in the following, we will adjust the proof of \cite[Proposition $5$]{DBLP:conf/compgeom/AichholzerOOPSS22} to the bipartite setting. 
\begin{theorem}\label{thm:small-k-bip}
    A bipartite outer $k$-planar graph with $n$ vertices and $k \leq 4$ has at most 
    \[\frac{(k+3.5)n - (2k+5)}{2}\]
  edges.
\end{theorem}
\begin{proof}
    As the proof is almost analogous to the one in \cite{DBLP:conf/compgeom/AichholzerOOPSS22}, we will only highlight the key difference.
    Let $G'$ be a maximal outerplanar subgraph of $G$. \cite{DBLP:conf/compgeom/AichholzerOOPSS22} establish
    \[
    2|E[G]| \leq k(n-2)-n+3+4(n-2)+|F(G')|
    \] where $F(G')$ denotes the set of faces of $G'$.
    Observe that since $G$ is bipartite, so is $G'$. 
    It is easy to show that $|F(G')| \leq 0.5n$ holds (where equality is achieved if every internal face is a quadrangle).
    Hence, 
     \[
    |E[G]| \leq \frac{(k+3.5)n - (2k+6)}{2}
    \] 
    as desired.
\end{proof}
This yields upper bounds of $\{1.75n-3, 2.25n-4, 2.75n-5, 3.25n-6, 3.75n-7\}$ for $k \in \{0,1,2,3,4\}$, respectively.
Using these bounds, we can tailor the Crossing Lemma to incorporate convexity and bipartiteness:
\begin{lemma}
Let $G$ be a bipartite graph with $n$ vertices and $m$ edges such that $m \geq 3.75n$. A convex geometric drawing of $G$ has at least
\[cr_o(G) \geq \frac{64}{675}\frac{m^3}{n^2}\]
crossings.
\end{lemma}
\begin{theorem}[(Second bipartite variant)]\label{thm:edge-density-second-bipartite}
Every bipartite outer $k$-planar graph $G$ on $n$ vertices has at most $\sqrt{\frac{675}
{128}k}n \approx 2.296\sqrt{k}n$ edges.
\end{theorem}

\subsection{The local approach}

\begin{theorem}\label{thm:bipartite-density}
For sufficiently large $k$, the minimum degree in a bipartite outer $k$-planar graph is at most $2\sqrt{\frac{8}{11}}\sqrt{k}+2$.
\end{theorem}
\begin{proof}
Throughout the proof, let $x = \sqrt{\frac{8}{11}}\approx 0.852$. Suppose for a contradiction that there exists a bipartite outer $k$-planar graph $G$ with minimum degree strictly larger than $2x\sqrt{k}+2$. This implies that every vertex is incident to more than $2x\sqrt{k}$ diagonals.
Again, we call a diagonal long if its length is at least $x\sqrt{k}$ and denote by $D$ a shortest long diagonal. We will show that $D$ is intersected by at least $k+1$ other diagonals.
Again, let $v_1,\dots,v_l$ be the vertices enclosed by $D$.
According to \cref{lem:maxcut}, at most $l(\frac{5}{8}x\sqrt{k}+76)$ edges are contained in the subgraph induced by $G[\{v_1,v_2,\dots,v_l\}]$, as by our choice of $D$, all long diagonals incident to $v_1,\dots,v_l$ necessarily cross $D$. This implies that $D$ is crossed by at least
\[
l(2x\sqrt{k})- (l-2) - l(\frac{5}{8}x\sqrt{k}+76)
\]
edges, where the second term accounts for the edges of $v_i$ incident to one of the endpoints of $D$ (restricted to the bipartite setting).
Hence, $D$ has at least
\[l(\frac{11x}{8}\sqrt{k}) -2-77l\]
crossings.
Recall that since  $l > x\sqrt{k}$, say $l = (1+\varepsilon)x\sqrt{k}$, this amounts to at least
\[((1+\varepsilon)x\sqrt{k})(\frac{11x}{8}\sqrt{k}) -2 -77((1+\varepsilon)x\sqrt{k}) > k\] edges crossing $D$ for sufficiently large $k$, as 
\[(1+\varepsilon)x\sqrt{k})(\frac{11x}{8}\sqrt{k}) = (1+\frac{11x\varepsilon}{8}) k \]
by our choice of $x$ which concludes the proof.
\end{proof}
\textit{Remark:} 

If one could improve the bound of \cref{lem:maxcut} to roughly $\frac{r}{2}n$ (which would be best possible without incorporating boundary conditions), then \cref{thm:bipartite-density} would yield a bound of $2 \sqrt{\frac{2}{3}}\sqrt{k} \approx 1.633\sqrt{k}$ for the maximum minimum degree, while our best construction (see \cref{thm:lower-bip}) has maximum minimum degree $\sqrt{2}\sqrt{k} \approx 1.41\sqrt{k}$.
\begin{corollary}[(Third bipartite variant)]\label{thm:edge-density-third-bipartite}
For sufficiently large $k$, every bipartite outer $k$-planar graph $G$ on $n$ vertices has at most $2\sqrt{\frac{8}{11}}\sqrt{k}n \approx 1.7\sqrt{k}n$ edges.
\end{corollary}

\subsection{The direct approach}
It is tempting to use the result of \cref{lem:maxcut} to bound the number of short diagonals which are contained in $G_1$, refer to the proof of \cref{thm:upper-new}. Unfortunately, this yields a worse bound than \cref{thm:upper-new}, namely it would yield $(\frac{13}{8}+\varepsilon)n > (\sqrt{2}+\varepsilon)n$ due to the fact that the boundary conditions are not considered in \cref{lem:maxcut}.

\section{Conclusions and Open Problems}\label{sec:open-problems}
The most natural open problem is to try and close the gap between the lower and upper-bound for the general case as well as for the bipartite case.
Concerning the lower bound construction given in \cref{thm:lower-general}, the authors of \cite{ábrego2024bookcrossingnumberscomplete} state that they do not believe this construction to be optimal when $k\geq 7$ holds. Still, we believe that the lower bound is significantly closer to the real bound than the current best upper bound due to \cref{cor:best-bound}.
For the bipartite case, all results (in particular the last two) of our approaches are subsumed by the result of \cref{cor:best-bound} for general graphs. In order to obtain a stricter bound for the bipartite case (using our approaches), one has to find a better solution to \cref{prob:one} by improving the result of \cref{lem:maxcut} and/or tackling \cref{prob:one} directly by including the boundary conditions.
Finally, it would also be of interest to see if the result of~\cite{DBLP:conf/stoc/Goncalves05}, which established that the edge-set of any planar graph can be decomposed into two outerplanar graphs, can be generalized to $k$-planarity in the following sense: Can the edge-set of a $k$-planar graph be decomposed into two sets which each induce an outer $k$-planar graph? If this could be answered in the affirmative, then \cref{cor:best-bound} would imply an upper bound of $2\sqrt{2}\approx2.828\sqrt{k}n$ edges (for sufficiently large $k$), while the current best bound is $\approx3.71\sqrt{k}n$~\cite{BestConstantCr}.


\bibliographystyle{splncs03}
\bibliography{references}


\end{document}